\documentclass[final,oneeqnum,onethmnum,onefignum,onetabnum]{siamltex1213}
\usepackage{amsfonts}
\usepackage{amsmath}
\usepackage{graphicx,subfigure,epsfig}
\usepackage{amssymb}
\usepackage{mathrsfs}
\usepackage{color}
\def \ProbUnion{P\left(\bigcup_{i=1}^N A_i\right)}
\def \lYAT{\ell_{\textrm{NEW}}}
\def \lKAT{\ell_{\textrm{KAT}}}
\def \lOPT{\ell_{\textrm{OPT}}}
\def \st{\textrm{s.t. }}

\newtheorem{remark}{Remark}
\newtheorem{example}{Example}

\title{Lower Bounds on the Probability of a Finite Union of Events\thanks{Parts of this work were presented at the 2014 IEEE International Symposium on Information Theory (ISIT'14). This work was supported in part by NSERC of Canada.}} 

\author{Jun Yang\footnotemark[2] \and Fady Alajaji\footnotemark[3] \and Glen Takahara\footnotemark[3]}

\begin{document}
\maketitle

\renewcommand{\thefootnote}{\fnsymbol{footnote}}

\footnotetext[2]{Department of Statistical Sciences, University of Toronto, Toronto, ON M5S3G3, Canada. E-mail: jun@utstat.toronto.edu.}
\footnotetext[3]{Department of Mathematics and Statistics, Queen’s University, Kingston, ON K7L3N6, Canada. E-mails: \{fady,takahara\}@mast.queensu.ca.}

\renewcommand{\thefootnote}{\arabic{footnote}}

\begin{abstract}
	In this paper, lower bounds on the probability of a finite union of events are considered, \emph{i.e.} $P\left(\bigcup_{i=1}^N A_i\right)$, in terms of the individual event probabilities $\{P(A_i), i=1,\ldots,N\}$ and the \emph{sums} of the pairwise event probabilities, \emph{i.e.}, $\{\sum_{j:j\neq i} P(A_i\cap A_j), i=1,\ldots,N\}$. The contribution of this paper includes the following: (i) in the class of all lower bounds that are established in terms of only the $P(A_i)$'s and $\sum_{j:j\neq i} P(A_i\cap A_j)$'s, the \emph{optimal} lower bound is given numerically by solving a linear programming (LP) problem with $N^2-N+1$ variables; (ii) a new analytical lower bound is proposed based on a relaxed LP problem, which is at least as good as the bound due to Kuai, et al. \cite{Kuai2000}; (iii) numerical examples are provided to illustrate the performance of the bounds.
\end{abstract}



\pagestyle{myheadings}
\thispagestyle{plain}
\markboth{Jun Yang, Fady Alajaji and Glen Takahara}{Lower Bounds on the Probability of a Finite Union of Events}

\section{Introduction}

Lower and upper bounds of $P\left(\bigcup_{i=1}^N A_i\right)$ in terms of the individual event probabilities $P(A_i)$'s and the pairwise event probabilities $P(A_i\cap A_j)$'s can be seen as special cases of the Boolean probability bounding problem \cite{Boros2011,Vizvari2004}, which can be solved numerically via a linear programming (LP) problem involving $2^N$ variables. Unfortunately, the number of variables for Boolean probability bounding problems increases exponentially with the number of events, $N$, which makes finding the solution impractical. Therefore, some suboptimal numerical bounds are proposed \cite{Boros2011,Vizvari2004,Prekopa2005,GalambosBook} in order to reduce the complexity of the LP problem, for example, by using the dual basic feasible solutions.

On the other hand, analytical lower bounds are particularly important. The Kuai-Alajaji-Takahara (KAT) bound \cite{Kuai2000} is one of the analytical lower bounds that has been shown to be better than the Dawson-Sankoff (DS) bound \cite{Dawson1967} and D. de Caen's bound \cite{DeCaen1997}. The KAT bound is extended in \cite{Prekopa2005} using sums of joint probabilities of up to $m$ events, where $m < N$, such as $\{\sum_{j,l}P(A_i\cap A_j\cap A_l), i=1,\dots,N\}$. These analytical bounds are later investigated in other works (e.g., see \cite{Chen2006,Hoppe2006,Hoppe2009,Kuai2000a,Behnamfar2005,Behnamfar2007}).

As in \cite{DeCaen1997}, the KAT lower bound \cite{Kuai2000} for $P\left(\bigcup_{i=1}^N A_i\right)$ is expressed in terms of only $\sum_{j:j\neq i} P(A_i\cap A_j)$'s and $P(A_i)$'s, and hence knowledge of the individual pairwise event probabilities $P(A_i\cap A_j)$ is not required. In this paper, we revisit and investigate the same problem that lower bounds are established in terms of only the \emph{sums} of the pairwise event probabilities, \emph{i.e.}, $\sum_{j:j\neq i} P(A_i\cap A_j)$, and the individual event probabilities $P(A_i)$'s, without the use of the $P(A_i\cap A_j)$'s.

Our contributions are the following. First, in the class of all lower bounds that are expressed in terms of only the $P(A_i)$'s and the $\sum_{j:j\neq i} P(A_i\cap A_j)$'s, the \emph{optimal} lower bound is obtained numerically by solving an LP problem, which has only $N^2-N+1$ variables. Here \emph{optimality} means that any lower bound for $P\left(\bigcup_{i=1}^N A_i\right)$ in terms of only $\sum_{j:j\neq i} P(A_i\cap A_j)$'s and $P(A_i)$'s cannot be sharper than the proposed lower bound. This is proven by showing that the proposed lower bound can always be achieved by constructing $\{A_i, i=1,\ldots,N\}$ that satisfy all known information on the $\sum_{j:j\neq i} P(A_i\cap A_j)$'s and $P(A_i)$'s. The computational complexity of the optimal lower bound is significantly improved since the number of variables is quadratic in $N$ (as opposed to being exponential in $N$). Next, a suboptimal analytical lower bound is established by solving a relaxed LP problem. The new analytical bound is proven to be at least as good as the existing KAT bound \cite{Kuai2000}. Finally, we analyze the performance of the new bounds by comparing them with the KAT bound and other existing bounds. In particular, numerical results show that the Gallot-Kounias (GK) bound \cite{Gallot1966,Kounias1968}, which was recently revisited in \cite{Feng2010,Mao2013}, is not necessarily sharper than the proposed lower bounds as well as the KAT bound (see also \cite{Feng2013} for another example), even though it exploits full information of all $P(A_i\cap A_j)$'s and $P(A_i)$'s. Furthermore, the Pr{\'e}kopa-Gao (PG) bound \cite{Prekopa2005}, which extends the KAT bound by using the additional partial information $\{\sum_{j,l}P(A_i\cap A_j\cap A_l), i=1,\dots,N\}$, is not necessarily tighter than the derived lower bounds.

\section{Main Results}

Consider a finite family of events $A_1,\ldots,A_N$ in a general probability space $(\Omega ,\mathscr{F},P)$, where $N$ is a fixed positive integer. Note that there are only finitely many Boolean atoms\footnote{The problem can be directly reduced to the finite probability space case. Thus, through the numerical examples in this paper, we will consider finite probability spaces where $\omega\in\Omega$ denotes an elementary outcome instead of an atom.} specified by the $A_i$'s \cite{DeCaen1997}. For each atom $\omega\in\mathscr{F}$, let $p(\omega):= P(\omega)$, and let the degree of $\omega$, denoted by $\deg(\omega)$, be the number of $A_i$'s that contain $\omega$. Define
\begin{equation}\label{def_aik}
a_i(k):= P(\{\omega\subseteq A_i: \deg(\omega)=k\}),
\end{equation}
where $i=1,\ldots,N$ and $k=1,\ldots,N$. Then from \cite[Lemma 1]{Kuai2000}, we know that
\begin{equation}\label{}
P\left(\bigcup_{i=1}^N A_i\right)=\sum_{i=1}^N\sum_{k=1}^N\frac{a_i(k)}{k}.
\end{equation}
In this paper, using the same notation as in \cite{Kuai2000}, lower bounds on $P\left(\bigcup_{i=1}^N A_i\right)$ are established only in terms of $\alpha_i:= P(A_i)$ and $\beta_i:= \sum_{j: j\neq i} P(A_i\cap A_j)$, $i=1,\ldots,N$. For simplicity, we denote $\gamma_i:=\alpha_i+\beta_i$. Then it is easy to verify that the following equalities hold:
\begin{equation}\label{}
P(A_i)=\sum_{k=1}^N a_i(k)=\alpha_i,\quad \sum_{j} P(A_i\cap A_j)=\sum_{k=1}^N k a_i(k)= \gamma_i,\quad i=1,\ldots,N.
\end{equation}

Let $\mathscr{L}$ denote the set of all lower bounds that are established in terms of only $\{\alpha_i, i=1,\ldots,N\}$ and $\{\gamma_i, i=1,\ldots,N\}$. Then any lower bound in $\mathscr{L}$, say $\ell\in\mathscr{L}$, is a function of only $\{\alpha_i\}$'s and $\{\gamma_i\}$'s. Also, claiming that $\ell\in\mathscr{L}$ is a lower bound on $P\left(\bigcup_{i=1}^N A_i\right)$ means that for any events $\{A_i, i=1,\ldots,N\}$ that satisfy $P(A_i)=\alpha_i,i=1,\ldots,N$ and $\sum_{j} P(A_i\cap A_j)=\gamma_i, i=1,\ldots,N$, we must have $P\left(\bigcup_{i=1}^N A_i\right)\ge \ell$.


In order to distinguish the use of different partial information, we assume that a vector $\theta=(\theta_1,\dots,\theta_m)\in\mathbb{R}^m$ represents partial probabilistic information about the union $\bigcup_{i=1}^N A_i$. Specifically, we assume that for a given integer $m\ge 1$, $\Theta$ denotes the range of a function of $P(A_i)$'s and $P(A_i\cap A_j)$'s, $\eta_m:[0,1]^{N+\binom{N}{2}}\rightarrow\mathbb{R}^m$. Then $\theta$ equals to the value of the function $\eta_m$ for given $A_1,\dots,A_N$. Then, we can define a lower bound of $\ProbUnion$ as a function of $\theta$, $\ell(\theta)$, such that $\ProbUnion\ge\ell(\theta)$ for any set of events $\{A_i\}$ that the value of $\eta_m$ for given $\{A_i\}$ equals to $\theta$.

Next, we define an \emph{optimal} lower bound in a general class of lower bounds that are functions of $\theta$. Let $\mathscr{L}_{\Theta}$ denote the set of all lower bounds on $P\left(\bigcup_{i=1}^N A_i\right)$ that are functions of only $\theta$.



\smallskip

\begin{definition}
	We say that a lower bound $\ell ^{\star}\in\mathscr{L}_{\Theta}$ is
	\textit{optimal in $\mathscr{L}_{\Theta}$} if $\ell ^{\star}(\theta)\ge\ell (\theta )$
	for all $\theta\in\Theta$ and $\ell\in\mathscr{L}_{\Theta}$.
\end{definition}

\smallskip

\begin{definition}
	We say that a lower bound $\ell\in\mathscr{L}_{\Theta}$ is
	\textit{achievable} if for every $\theta\in\Theta$,
	$$\inf_{A_1,\ldots ,A_N}P\left(\bigcup_{i=1}^N A_i\right)=\ell (\theta ),$$
	where the infimum ranges over all collections $\{A_1,\ldots ,A_N\}$,
	$A_i\in\mathscr{F}$, such that $\{A_1,\ldots ,A_N\}$ is represented by
	$\theta$.
\end{definition}

For bounds in $\mathscr{L}_{\Theta}$, the following lemma shows that achievability is equivalent to optimality.

\smallskip

\begin{lemma}\label{optimality_lemma}
	A lower bound $\ell ^{\star}\in\mathscr{L}_{\Theta}$ is optimal
	in $\mathscr{L}_{\Theta}$ if and only if it is achievable.
\end{lemma}

\smallskip

\begin{proof}
	Suppose that $\ell ^{\star}$ is achievable. Let $\theta\in\Theta$
	and $\epsilon >0$ be given, and let $\ell$ be any lower bound in
	$\mathscr{L}_{\Theta}$. By achievability there exist sets $A_1,\ldots ,A_N$
	in $\mathscr{F}$ represented by $\theta$ such that
	\[
	\ell ^{\star}(\theta )>P\left(\bigcup_{i=1}^N A_i\right)-\epsilon \ge\ell (\theta )-\epsilon .
	\]
	Since this holds for any $\epsilon$ we have
	$\ell ^{\star}(\theta )\ge\ell (\theta )$. We prove the converse by the
	contrapositive. Suppose that $\ell ^{\star}$ is not achievable. Then there
	exists $\theta '\in\Theta$ such that
	\[
	\inf_{A_1,\ldots ,A_N}P\left(\bigcup_{i=1}^N A_i\right)>\ell ^{\star}(\theta '),
	\]
	where the infimum ranges over all collections $\{A_1,\ldots ,A_N\}$,
	$A_i\in\mathscr{F}$, such that $\{A_1,\ldots ,A_N\}$ is represented by
	$\theta '$. Define $\ell$ by
	\[
	\ell (\theta ) =\left\{\begin{array}{cl}
	c & \mbox{if $\theta =\theta '$} \\
	0 & \mbox{if $\theta\ne\theta '$,}
	\end{array}\right.
	\]
	where $c$ satisfies
	\[
	\inf_{A_1,\ldots ,A_N}P\left(\bigcup_{i=1}^N A_i\right)>c>\ell ^{\star}(\theta ').
	\]
	Then $\ell\in\mathscr{L}_{\Theta}$ and is larger than $\ell ^{\star}$ at
	$\theta '$. Hence, $\ell ^{\star}$ is not optimal.
\end{proof}

Clearly, in our problem, we have $\theta=(\alpha_1,\ldots,\alpha_N, \gamma_1,\ldots,\gamma_N)$ and $\mathscr{L}_{\Theta}=\mathscr{L}$.
We herein state the following lemma regarding the existing KAT bound.
\begin{lemma}[KAT Bound~\cite{Kuai2000}]\label{lemma_KAT}
	The solution of the following LP problem
	\begin{equation}\label{LP_KAT}
	\begin{split}
	\min_{\{a_i(k),i=1,\ldots,N,k=1,\ldots,N\}}&\quad\sum_{i=1}^N\sum_{k=1}^N\frac{a_i(k)}{k}\\
	\st&\quad\sum_{k=1}^N a_i(k)=\alpha_i,\quad \sum_{k=1}^N k a_i(k)= \gamma_i,\quad i=1,\ldots,N,\\
	&\quad a_i(k)\ge 0,\quad i=1,\ldots,N,\quad k=1,\ldots,N,
	\end{split}
	\end{equation}
	gives the KAT bound:
	\begin{equation}\label{KAT_bound}
	P\left(\bigcup_{i=1}^N A_i\right)\ge \sum_{i=1}^N\quad \left\{ \left[\frac{1}{\lfloor\frac{\gamma_i}{\alpha_i}\rfloor}-\frac{\frac{\gamma_i}{\alpha_i}-\lfloor\frac{\gamma_i}{\alpha_i}\rfloor}{(1+\lfloor\frac{\gamma_i}{\alpha_i}\rfloor)(\lfloor\frac{\gamma_i}{\alpha_i}\rfloor)}\right]\alpha_i\right\},
	\end{equation}
	where $\lfloor x\rfloor$ is the largest positive integer less than or equal to $x$.
\end{lemma}

\smallskip

Denoting $\lKAT$ as the KAT bound in (\ref{KAT_bound}), we can see that the KAT bound is a lower bound which is established in terms of only $\{\alpha_i\}$'s and $\{\gamma_i\}$'s. Thus, $\lKAT\in\mathscr{L}$. One should note that for a given family of events $\{A_i, i=1,\ldots,N\}$, the $a_i(k)$'s can be obtained from their definition in (\ref{def_aik}). However, this does not mean that for each feasible point $\{a_i(k)\}$ of the LP problem (\ref{LP_KAT}), there exists a corresponding family of events $\{A_i, i=1,\ldots,N\}$ . In particular, for the solution of (\ref{LP_KAT}), it is possible that a family of events $\{A_i, i=1,\ldots,N\}$ can never be constructed; this is illustrated in the following example.

\medskip

\begin{example}\label{example1}
	Considering a finite probability space (where atoms $\omega$ are reduced to elementary outcomes), shown as System V in Table \ref{table_sysV}, we have $$N=3, \alpha_1=0.1, \alpha_2=\alpha_3=0.2, \gamma_1=0.21, \gamma_2=\gamma_3=0.265.$$  The KAT solution $\{a_i(k)\}$ for $\lKAT=0.3833$ is obtained only at the following optimal feasible point of (\ref{LP_KAT}):
	\begin{equation*}
	\begin{split}
	&a_1(1)=0, a_1(2)=0.09, a_1(3)=0.01, a_2(1)=a_3(1)=0.135, \\
	&a_2(2)=a_3(2)=0.065, a_2(3)=a_3(3)=0.
	\end{split}
	\end{equation*}
	However, $a_1(3):= P(\{\omega\in A_1: \deg(\omega)=3\})=0.01$ implies $P(A_1\cap A_2\cap A_3)\ge 0.01$, since $\deg(\omega)=3$ means that the corresponding outcome $\omega$ must be contained in all $A_i,i=1,\ldots,3$. However, $a_2(3):= P(\{\omega\in A_2: \deg(\omega)=3\})=0$ implies such $\omega$ is not in $A_2$, which is a contradiction. Therefore, there is no family of events $\{A_1,A_2,A_3\}$ that can be constructed for this system such that $P(A_1\cup A_2\cup A_3)=0.3833$. In other words, for any sets $\{A_1,A_2,A_3\}$ with given value of $\{\alpha_i\}$'s and $\{\gamma_i\}$'s, we must have $P(A_1\cup A_2\cup A_3)>0.3833$.
	\begin{table}\caption{System V.}\label{table_sysV}
		\centering
		\begin{tabular}{|c|c||c|c|c|}
			\hline
			Outcomes $\omega_i$ & $p(\omega_i)$ & $A_1$ & $A_2$ & $A_3$\\ \hline\hline
			$\omega_0$ & 0.145 &  &  & $\times$ \\ \hline
			$\omega_1$ & 0.045 & $\times$ &  & $\times$ \\ \hline
			$\omega_2$ & 0.01 & $\times$ & $\times$ & $\times$ \\ \hline
			$\omega_3$ & 0.045 & $\times$ & $\times$ &  \\ \hline
			$\omega_4$ & 0.145 &  & $\times$ & \\ \hline
		\end{tabular}
	\end{table}
\end{example}

\medskip

\begin{remark}\label{remark_KAT_unique_solution}
	It can be shown that the LP problem (\ref{LP_KAT}) has a unique optimal feasible point. Therefore, the KAT bound is achievable if and only if the optimal feasible point of the LP problem (\ref{LP_KAT}) has a corresponding family of events $\{A_i, i=1,\cdots,N\}$ that satisfies the information represented by $\theta=(\alpha_1,\cdots,\alpha_N,\gamma_1,\cdots,\gamma_N)$. From Example~\ref{example1}, we see that $\lKAT$ is not optimal in $\mathscr{L}$.
\end{remark}

\bigskip
\subsection{Optimal Numerical Lower Bound}
In order to get a better lower bound than the KAT bound, we herein introduce more constraints on the $a_i(k)$'s in (\ref{LP_KAT}) so that the feasible set of $a_i(k)$'s becomes smaller, thus resulting in a sharper lower bound. By Lemma~\ref{optimality_lemma}, if a family of events $\{A_i\}$ can always be constructed for any feasible point of the resulting LP problem, then the solution must be the optimal lower bound. We establish the numerically computable optimal lower bound in the following theorem.

\smallskip

\begin{theorem}[Optimal Numerical Lower Bound]\label{theorem1}
	The \emph{optimal} lower bound in $\mathscr{L}$ is given by solving the following LP problem:
	\begin{equation}\label{LP_optimal}
	\begin{split}
	\min_{\{a_i(k),i=1,\ldots,N,k=1,\ldots,N\}}&\quad\sum_{i=1}^N\sum_{k=1}^N\frac{a_i(k)}{k}\\
	\st&\quad\sum_{k=1}^N a_i(k)=\alpha_i,\quad \sum_{k=1}^N k a_i(k)= \gamma_i,\quad i=1,\ldots,N,\\
	&\quad \sum_{i=1}^N a_i(k)\ge k a_j(k),\quad j=1,\ldots,N,\quad k=1,\ldots,N,\\
	&\quad a_i(k)\ge 0,\quad i=1,\ldots,N,\quad k=1,\ldots,N,
	\end{split}
	\end{equation}
	where the number of variables can be reduced to $N^2-N+1$.
\end{theorem}

\smallskip

\begin{proof}
	Denote the optimal lower bound in $\mathscr{L}$ by $\lOPT$ then $\lOPT\in\mathscr{L}$ satisfies $\lOPT\ge \ell$ for all $\ell\in\mathscr{L}$. Let the solution of (\ref{LP_optimal}) be $\lOPT'$, we will show that $\lOPT'=\lOPT$. First, it is easy to prove that for any $\{a_i(k)\}$ obtained by (\ref{def_aik}) from a family of events $\{A_i\}$, the additional constraints $\sum_{i=1}^N a_i(k)\ge k a_j(k)$ must hold for each $j=1,\ldots,N$ and $k=1,\ldots,N$. Therefore, $\lOPT'$ is a lower bound on $P\left(\bigcup_{i=1}^N A_i\right)$. Furthermore, since $\lOPT'$ is established in terms of only $\{\alpha_i\}$'s and $\{\gamma_i\}$'s, we have $\lOPT'\in\mathscr{L}$. Thus, we only need to prove $\lOPT'\ge \ell'$ for all $\ell'\in\mathscr{L}$. Also, note that for the solution of (\ref{LP_optimal}), since $\lOPT'\le P\left(\bigcup_{i=1}^N A_i\right)\le 1$, the objective value must be no larger than $1$, i.e., $\sum_{i=1}^N\sum_{k=1}^N\frac{a_i(k)}{k}\le 1$. Thus, the optimal feasible point of (\ref{LP_optimal}) must fall into the subset of the feasible set of (\ref{LP_optimal}), which is determined by the additional constraint $\sum_{i=1}^N\sum_{k=1}^N\frac{a_i(k)}{k}\le 1$. In the following, we prove that $\lOPT'$ is achievable, i.e., a family of events $\{A_i\}$ can always be constructed from the solution of (\ref{LP_optimal}). Then, the optimality of $\lOPT'$ follows by Lemma \ref{optimality_lemma}.
	
	\emph{Achievability:} We prove that for any $\{a_i(k)\}$ that satisfies the constraints of (\ref{LP_optimal}) and the additional constraint $\sum_{i=1}^N\sum_{k=1}^N\frac{a_i(k)}{k}\le 1$, it is always possible to construct a family of events $\{A_i\}$ such that $P(\{\omega\subseteq A_i: \deg(\omega)=k\})=a_i(k)$ holds.
	The construction method is given as follows:
	\begin{itemize}
		\item The set $\Omega'$ is composed of $N\times N$ atoms, denoted as $\{\omega_i^{(k)}, i=1,\ldots,N, k=1,\ldots,N\}$. In the following, $\{\omega_i^{(k)}, i=1,\ldots,N\}$ are constructed separately for each $k$.
		\item Consider $N$ circles such that the $k$-th circle has a perimeter equals to $\sum_{i=1}^N \frac{a_i(k)}{k}$, $k=1,\ldots,N$. Then for the $k$-th circle, $\sum_{i=1}^N a_i(k)$ equals $k$ times its perimeter. Furthermore, since $a_j(k)\le\sum_{i=1}^N \frac{a_i(k)}{k}$ for all $j$, $a_j(k)$ is no larger than the perimeter of the $k$-th circle.
		\item For $j=1,\ldots,N$, we map the points on the arc of length $a_j(k)$ on the $k$-th circle from $2\pi\frac{k\sum_{l=1}^{j-1} a_l(k)}{\sum_{i=1}^N a_i(k)}$   to $2\pi\frac{k\sum_{l=1}^{j} a_l(k)}{\sum_{i=1}^N a_i(k)}$ to a set $B_j^{(k)}$. Then since for the $k$-th circle, $\sum_{i=1}^N a_i(k)$ equals to $k$ times its perimeter and $a_j(k)$ is no larger than its perimeter, it follows that every point on the $k$-th circle is mapped to exactly $k$ distinct sets in $\{B_1^{(k)},\ldots,B_N^{(k)}\}$.
		\item On the $k$-th circle, the points at the following $N$ angles,
		$$2\pi\left(\frac{k\sum_{l=1}^{j} a_l(k)}{\sum_{i=1}^N a_i(k)}-\left\lfloor \frac{k\sum_{l=1}^{j} a_l(k)}{\sum_{i=1}^N a_i(k)} \right\rfloor\right),\quad j=1,\ldots,N$$ divide the circle into (at most) $N$ arcs, and the points on each arc are mapped to the same $k$ sets in $\{B_1^{(k)},\ldots,B_N^{(k)}\}$. Let $\{\theta_j^{(k)}, j=1,\ldots,N\}$ be the ordered tuple of $$\left\{2\pi\left(\frac{k\sum_{l=1}^{j} a_l(k)}{\sum_{i=1}^N a_i(k)}-\left\lfloor \frac{k\sum_{l=1}^{j} a_l(k)}{\sum_{i=1}^N a_i(k)} \right\rfloor\right), j=1,\ldots,N\right\},$$then $0= \theta_1^{(k)}\le \theta_2^{(k)}\le\ldots\le\theta_N^{(k)}\le 2\pi$. Construct the atom $\omega_j^{(k)}$ such that its probability $p(\omega_j^{(k)}$) equals to the length of the $j$-th arc of the $k$-th circle, i.e.,
		\begin{equation}\label{}
		p(\omega_j^{(k)})=\left\{ \begin{array}{ll}
		(\theta_{j+1}^{(k)}-\theta_{j}^{(k)})\frac{\sum_{i=1}^N a_i(k)}{2\pi k} & \textrm{for $j<N$,}\\
		(2\pi-\theta_{N}^{(k)})\frac{\sum_{i=1}^N a_i(k)}{2\pi k} & \textrm{for $j=N$.}
		\end{array} \right.
		\end{equation}
		\item Since the points on the $j$-th arc are mapped to $k$ sets $\{B_{i_{1j}}^{(k)},\ldots,B_{i_{kj}}^{(k)}\}$ where $\{i_{1j},\ldots,i_{kj}\}\in\{1,\ldots,N\}$ contains $k$ different numbers, we let the atom $\omega_j^{(k)}$ be a subset of $A_{i_{1j}},\ldots,A_{i_{kj}}$, respectively, i.e., $\omega_j^{(k)}\subseteq A_{i_{1j}}\cap\ldots\cap A_{i_{kj}}$.
		\item For each $k$, the total probability of all constructed atoms equals to the perimeter of the circle, $\sum_{i=1}^N\frac{a_i(k)}{k}$. Also, each atom $\omega_j^{(k)}$ is contains in exactly $k$ events of  $A_1,\ldots,A_N$.
		Finally, since there are in total $N\times N$ atoms $\{\omega_j^{(k)}, j=1,\ldots,N, k=1,\ldots,N\}$, each constructed $A_i$ contains a finite number of atoms.
	\end{itemize}
	With the construction described above, it can be readily checked that the constructed $\{A_i\}$ satisfy $P(\{\omega\subseteq A_i: \deg(\omega)=k\})=a_i(k)$ for all $i=1,\ldots,N$. Since $\lOPT'$ is achieved at one feasible point of (\ref{LP_optimal}), by the proposed construction method a family of events, say $\{A_i^*\}$, can be constructed so that $P\left(\bigcup_{i=1}^N A_i^*\right)=\lOPT'$. Since the first two constraints of (\ref{LP_optimal}) are also satisfied, we have $P(A_i^*)=\alpha_i$ and $\sum_{j}P(A_i^*\cap A_j^*)=\gamma_i$ for all $i$.
	
	Therefore, the optimality of $\lOPT'$ directly follows by Lemma \ref{optimality_lemma}.
	Finally, the number of variables in bound~(\ref{LP_optimal}) can be reduced from $N^2$ to $N^2-N+1$ by observing that $a_1(N)=a_2(N)=\ldots=a_N(N)$. 
\end{proof}

\medskip

\begin{example} We give an example in a finite probability space to illustrate the construction provided in the achievability part of the above proof for $N=4$ and $k=2$. Assume that
	$a_1(k)=0.1$, $a_2(k)=0.2$, $a_3(k)=0.3$, and $a_4(k)=0.4$ for $k=2$. Since $a_j(k)\le\sum_{i=1}^4 \frac{a_i(k)}{k}=0.5$ hold for $j=1,\ldots,4$, the given $a_j(k)$'s satisfy the constraints in (\ref{LP_optimal}) for $k=2$.
	\begin{itemize}
		\item In order to construct the outcomes, we assume there is a circle with perimeter equals to $\sum_{i=1}^4 \frac{a_i(k)}{k}=0.5$. Then we map the arc $(0,0.4\pi]$ to $B_1^{(2)}$, $(0.4\pi,1.2\pi]$ to $B_2^{(2)}$, $(1.2\pi,2.4\pi]$ to $B_3^{(2)}$, and $(2.4\pi,4\pi]$ to $B_4^{(2)}$, as shown in Fig. \ref{fig_example}. Then every arc generates an angle less than $2\pi$ and every point on the circle is mapped to exactly two sets in $\{B_1^{(2)},B_2^{(2)},B_3^{(2)},B_4^{(2)}\}$. That is: the arc $(0,0.4\pi]$ is mapped to $B_1^{(2)}$ and $B_3^{(2)}$; the arc $(0.4\pi,1.2\pi]$ is mapped to $B_2^{(2)}$ and $B_4^{(2)}$; the arc $(1.2\pi, 2\pi]$ is mapped to $B_3^{(2)}$ and $B_4^{(2)}$.
		\item Since the ordered tuple of the angles $\{0.4\pi, 1.2\pi, 2\pi(1.2-1), 2\pi(2-2)\}$ is $\{0,0.4\pi,0.4\pi,1.2\pi\}$, the circle is divided by $N=4$ arcs with lengths equal to $\{0.1, 0, 0.2, 0.2\}$, respectively.
		\item The outcomes $\omega_1^{(2)}$, $\omega_2^{(2)}$, $\omega_3^{(2)}$ and $\omega_4^{(2)}$ are constructed with probabilities equal to the length of the arcs, i.e., $p(\omega_1^{(2)})=0.1$, $p(\omega_2^{(2)})=0$, $p(\omega_3^{(2)})=0.2$, $p(\omega_3^{(2)})=0.2$. Finally, we set the outcomes belonging to events $A_i$'s as follows: $\omega_1^{(2)}\in A_1\cap A_3$, $\omega_3^{(2)}\in A_2\cap A_4$ and $\omega_4^{(2)}\in A_3\cap A_4$. After the construction for $k=2$ only, the events of $\{A_i\}$ become: $A_1=\{\omega_1^{(2)}\}$, $A_2=\{\omega_3^{(2)}\}$, $A_3=\{\omega_1^{(2)},\omega_4^{(2)}\}$, $A_4=\{\omega_3^{(2)},\omega_4^{(2)}\}$. Thus, $P(\{\omega\in A_i: \deg(\omega)=k\})=a_i(k)$ is satisfied for $i=1,\ldots,4$ and $k=2$.
	\end{itemize}
	\begin{figure}
		\centering\includegraphics[width=1\textwidth]{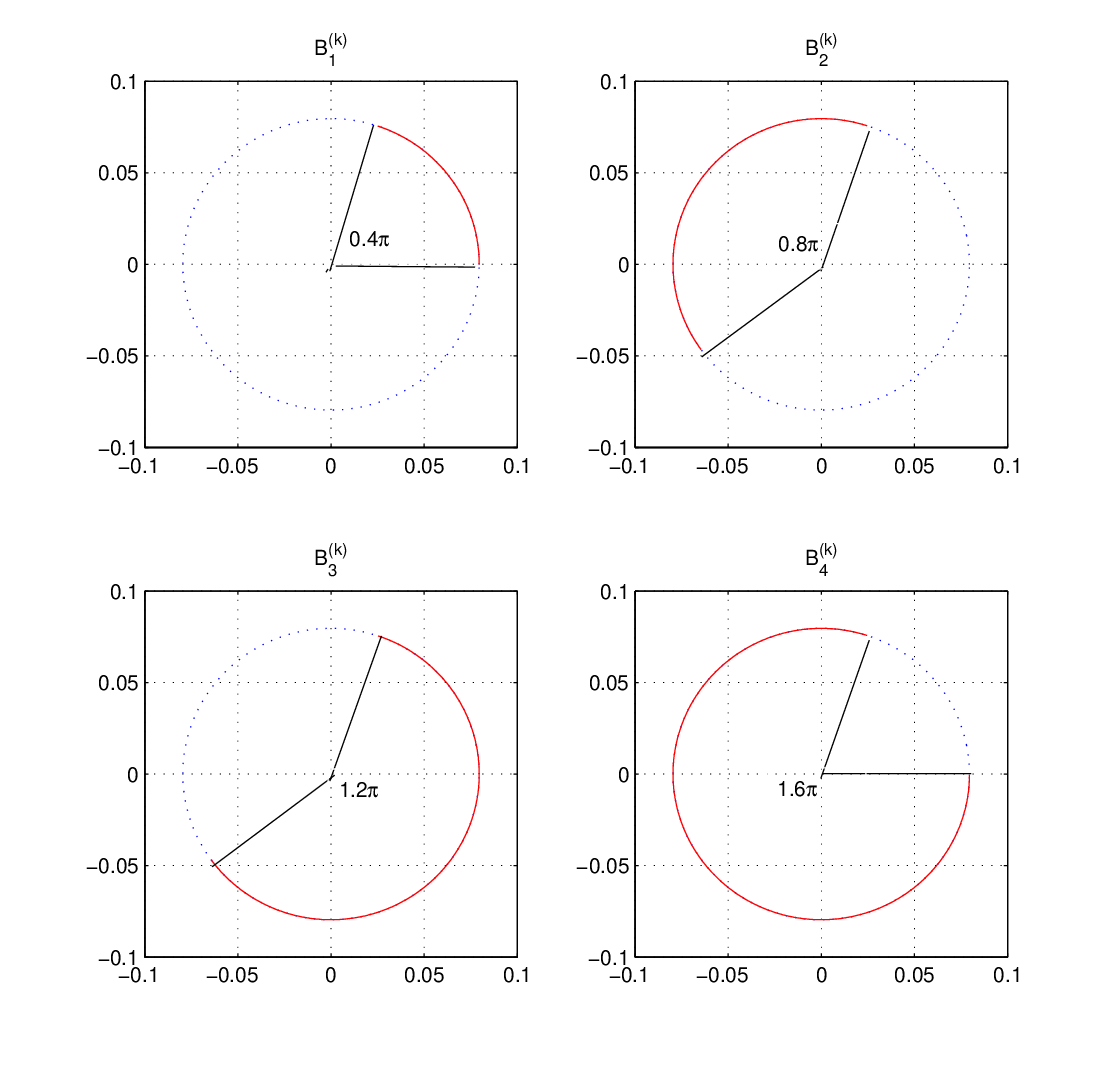}
		\caption{Example illustrating the construction of the proof of Theorem \ref{theorem1} for $N=4$ and $k=2$.}\label{fig_example}
	\end{figure}
\end{example}

\medskip

\begin{remark}
	The existing DS bound \cite{Dawson1967} is known to be optimal in the class of lower bounds with the information $\theta=(S_1=\sum_{i=1}^N P(A_i), S_2=\sum_{i,j} P(A_i\cap A_j))$ (e.g., see \cite[p.~22]{GalambosBook}). Using Lemma \ref{optimality_lemma}, we can provide a different proof of the optimality of the DS bound \cite{Dawson1967}. Specifically, we can show that the DS bound is the solution of the following LP problem:
	\begin{equation}\label{LP_dawson}
	\begin{split}
	\min_{\{a_i(k)\}}&\quad\sum_{i=1}^N\sum_{k=1}^N\frac{a_i(k)}{k}\\
	\st&\quad\sum_{i=1}^N\sum_{k=1}^N a_i(k)=S_1,\quad \sum_{i=1}^N\sum_{k=1}^N k a_i(k)=S_2,\\
	&\quad \sum_{i=1}^N a_i(k)\ge k a_j(k),\quad j=1,\ldots,N,\quad k=1,\ldots,N,\\
	&\quad a_i(k)\ge 0,\quad i=1,\ldots,N,\quad k=1,\ldots,N.
	\end{split}
	\end{equation}
	The last two constraints in (\ref{LP_dawson}) together with $\sum_{i=1}^N\sum_{k=1}^N\frac{a_i(k)}{k}\le 1$ guarantee the achievability of the solution of (\ref{LP_dawson}). Thus, by Lemma \ref{optimality_lemma}, the solution of (\ref{LP_dawson}) is the optimal lower bound.
\end{remark}

\bigskip
\subsection{New Analytical Lower Bound}
We herein derive a new analytical lower bound in $\mathscr{L}$, 
which is given in the following theorem.

\smallskip

\begin{theorem}[New Analytical Bound]\label{theorem2}
	The lower bound is given by
	\begin{equation}\label{YA_bound}
	P\left(\bigcup_{i=1}^N A_i\right)\ge \lYAT:=\delta+\sum_{i=1}^N\quad \left\{ \left[\frac{1}{\chi(\frac{\gamma_i'}{\alpha_i'})}-\frac{\frac{\gamma_i'}{\alpha_i'}-\chi(\frac{\gamma_i'}{\alpha_i'})}{[1+\chi(\frac{\gamma_i'}{\alpha_i'})][\chi(\frac{\gamma_i'}{\alpha_i'})]}\right]\alpha_i'\right\},
	\end{equation}
	where the function $\chi(\cdot)$ is defined by
	\begin{equation}\label{definition_chi}
	\chi(x) = \left\{ \begin{array}{ll}
	n-1 & \textrm{if $x=n$ where $n\ge 2$ is a integer}\\
	\lfloor x\rfloor & \textrm{otherwise}
	\end{array} \right.
	\end{equation}
	and
	\begin{equation}\label{}
	\delta:=\left\{\max_i\left[\gamma_i-(N-1)\alpha_i\right]\right\}^+\ge 0, \quad \alpha_i':= \alpha_i-\delta,\quad \gamma_i':= \gamma_i-N \delta.
	\end{equation}
	
\end{theorem}
\begin{proof}
	The new lower bound is the solution of the following relaxed LP:
	\begin{equation}\label{LP_analytical}
	\begin{split}
	\min_{\{a_i(k),i=1,\ldots,N,k=1,\ldots,N\}}&\quad\sum_{i=1}^N\sum_{k=1}^N\frac{a_i(k)}{k}\\
	\st&\quad\sum_{k=1}^N a_i(k)=\alpha_i,\quad \sum_{k=1}^N k a_i(k)= \gamma_i,\quad i=1,\ldots,N,\\
	&\quad \sum_{i=1}^N a_i(N)\ge N a_j(N),\quad j=1,\ldots,N\\
	&\quad a_i(k)\ge 0,\quad i=1,\ldots,N,\quad k=1,\ldots,N.
	\end{split}
	\end{equation}
	Note that the above problem is a relaxed problem of (\ref{LP_optimal}) because the constraints $\sum_{i=1}^N a_i(k)\ge k a_j(k), j=1,\ldots,N$ for all $k\neq N$ in (\ref{LP_optimal}) are relaxed. Comparing with the LP problem of (\ref{LP_KAT}) that corresponds to the KAT bound, the additional constraints are only $\sum_{i=1}^N a_i(N)\ge N a_j(N), j=1,\ldots,N$, which can be easily proved to be equivalent to requiring that $a_1(N)=a_2(N)=\ldots=a_N(N)$. We first introduce a new non-negative variable $x:= a_1(N)=\ldots=a_N(N)$, and solve the problem (\ref{LP_analytical}) by assuming that $x$ is known. Then the objective function in  (\ref{LP_analytical}) becomes a function of $x$. Finally, we minimize the objective function to yield a solution of (\ref{LP_analytical}).
	
	Replacing $a_i(N), i=1,\ldots,N$ in (\ref{LP_analytical}) by $x$ and assuming that $x$ is given implies that (\ref{LP_analytical}) can be solved separately for each $i, i=1,\ldots,N$, by solving the following $N$ problems:
	\begin{equation}\label{separate_problem}
	\begin{split}
	f_i(x):=\min_{a_i(k), k=1,\ldots,N-1}&\quad\sum_{k=1}^{N-1} \frac{a_i(k)}{k}+ \frac{x}{N}\\
	\st&\quad \sum_{k=1}^{N-1}a_i(k)=\alpha_i-x\\
	&\quad \sum_{k=1}^{N-1} k a_i(k)=\gamma_i-Nx\\
	&\quad a_i(k)\ge 0, k=1,\ldots, N-1.
	\end{split}
	\end{equation}
	Note that when $x$ is given the above problem is equivalent to (\ref{LP_KAT}), the solution of which was derived in different ways in \cite{Kuai2000,KuaiThesis}. However, since $x$ is a variable which is assumed to be fixed at the current stage, the solution of problem (\ref{separate_problem}) may not exist for any given $x$. Thus, one needs to investigate the condition for the existence of a solution for (\ref{separate_problem}) when solving it. To this end, we solve the problem (\ref{separate_problem}) by taking into account the feasible set for $x$.
	
	Since the LP problem (\ref{separate_problem}) has $N-1$ variables and the LP optimum must be achieved at one of vertices of the polyhedron formed by the constraints \cite{bertsimas-LPbook}, the $N-3$ of the $N-1$ constraints $a_i(k)\ge 0, k=1,\ldots,N-1$ must be active. Assume that the other two constraints $a_i(k)\ge 0$ that are not active are given for $k=k_1$ and $k=k_2$ and $1\le k_1< k_2\le N-1$, then we obtain
	\begin{equation}\label{}
	a_i(k_1)+a_i(k_2)=\alpha_i-x,\quad k_1a_i(k_1)+k_2a_i(k_2)=\gamma_i-Nx,
	\end{equation}
	which yields
	\begin{equation}\label{}
	a_i(k_1)=\frac{k_2(\alpha_i-x)-(\gamma_i-Nx)}{k_2-k_1}\ge 0,\quad a_i(k_2)=\frac{(\gamma_i-Nx)-k_1(\alpha_i-x)}{k_2-k_1}\ge 0.
	\end{equation}
	Using the condition $1\le k_1<k_2\le N-1$, the solution exists when
	\begin{equation}\label{existence_condition}
	[\gamma_i-(N-1)\alpha_i]^+\le x\le \frac{\beta_i}{N-1},
	\end{equation}
	and $k_2$ and $k_1$ satisfy $k_1\le\frac{\gamma_i-Nx}{\alpha_i-x}\le k_2$.
	
	Next, we prove that $\frac{a_i(k_1)}{k_1}+\frac{a_i(k_2)}{k_2}$ is non-decreasing with $k_2$ and non-increasing with $k_1$. Let $k:=\frac{\gamma_i-Nx}{\alpha_i-x}$, then
	\begin{equation}\label{}
	a_i(k_1)=\frac{k_2(\alpha_i-x)-(\gamma_i-Nx)}{k_2-k_1}
	=(\alpha_i-x)\frac{k_2-\frac{\gamma_i-Nx}{\alpha_i-x}}{k_2-k_1}=(\alpha_i-x)\frac{k_2-k}{k_2-k_1}.
	\end{equation}
	Similarly, we have
	\begin{equation}\label{}
	a_i(k_2)=\frac{(\gamma_i-Nx)-k_1(\alpha_i-x)}{k_2-k_1}=(\alpha_i-x)\frac{k-k_1}{k_2-k_1}.
	\end{equation}
	Since $\alpha_i-x$ is a constant here and $k_1\le k\le k_2$, we only need to consider
	\begin{equation}\label{}
	\begin{split}
	\frac{1}{\alpha_i-x}\left[\frac{a_i(k_1)}{k_1}+\frac{a_i(k_2)}{k_2}\right]&=\frac{1}{k_1}\frac{k_2-k}{k_2-k_1}+\frac{1}{k_2}\frac{k-k_1}{k_2-k_1}\\
	&=\frac{1}{k_2-k_1}\left[\left(\frac{k_2}{k_1}-\frac{k}{k_1}\right)+\left(\frac{k}{k_2}-\frac{k_1}{k_2}\right)\right]\\
	&=\frac{1}{k_2-k_1}\left(\frac{k_2^2-k_1^2}{k_1k_2}+\frac{k_1-k_2}{k_1k_2}k\right)\\
	&=\frac{1}{k_1}+\frac{1}{k_2}-\frac{k}{k_1k_2}
	\end{split}
	\end{equation}
	The derivatives w.r.t. $k_1$ and $k_2$ can be obtained as follows:
	\begin{equation}\label{}
	\frac{1}{k_1^2}\left(\frac{k}{k_2}-1\right)\le 0, \quad \frac{1}{k_2^2}\left(\frac{k}{k_1}-1\right)\ge 0.
	\end{equation}
	Therefore, we have shown $\frac{a_i(k_1)}{k_1}+\frac{a_i(k_2)}{k_2}$ is non-increasing with $k_1$ and non-decreasing with $k_2$. As a result, the optimal $k_2$ and $k_1$ when $\frac{\gamma_i-Nx}{\alpha_i-x}$ is not an integer must be
	\begin{equation}\label{expression_of_k1_k2}
	k_1=\left\lfloor\frac{\gamma_i-Nx}{\alpha_i-x}\right\rfloor,\quad k_2=k_1+1.
	\end{equation}
	
	When $\frac{\gamma_i-Nx}{\alpha_i-x}$ is an integer, one can choose either $k_1=\frac{\gamma_i-Nx}{\alpha_i-x}, k_2=k_1+1$ or $k_1=\frac{\gamma_i-Nx}{\alpha_i-x}-1, k_2=k_1+1$, since for both cases the values of $\frac{a_i(k_1)}{k_1}+\frac{a_i(k_2)}{k_2}$ are indeed identical. Note that the condition for the existence of the solution to (\ref{LP_analytical}) implies that $1\le k_1<k_2\le N-1$, thus, the optimal $k_1$ and $k_2$ that give the largest feasible set of $x$ are
	\begin{equation}\label{optimal_k1_k2}
	k_1=\chi(\frac{\gamma_i-Nx}{\alpha_i-x}),\quad k_2=k_1+1.
	\end{equation}
	Then the solution of (\ref{separate_problem}) which is a function of $x$ can be written as
	\begin{equation}\label{def_f_i_x}
	\begin{split}
	f_i(x)=&\frac{2\chi(\frac{\gamma_i-Nx}{\alpha_i-x})+1}{\chi(\frac{\gamma_i-Nx}{\alpha_i-x})\left[\chi(\frac{\gamma_i-Nx}{\alpha_i-x})+1\right]}(\alpha_i-x)\\
	&-\frac{1}{\chi(\frac{\gamma_i-Nx}{\alpha_i-x})\left[\chi(\frac{\gamma_i-Nx}{\alpha_i-x})+1\right]}(\gamma_i-Nx)+\frac{x}{N},
	\end{split}
	\end{equation}
	where $ [\gamma_i-(N-1)\alpha_i]^+\le x\le \frac{\beta_i}{N-1}$.

	Next, we prove that $f_i(x)$ is a non-decreasing function of $x$. First, we prove that the function $f_i(x)$ is continuous. Note that by definition of $\gamma_i$ and $\alpha_i$, we know $\gamma_i\le N\alpha_i$.
	\begin{equation}\label{}
	\begin{split}
	\left(\frac{\gamma_i-Nx}{\alpha_i-x}\right)'&=\frac{(-N)(\alpha_i-x)-(\gamma_i-Nx)(-1)}{(\alpha_i-x)^2}\\
	&=\frac{(\gamma_i-Nx)-N(\alpha_i-x)}{(\alpha_i-x)^2}\\
	&\le\frac{(N\alpha_i-Nx)-N(\alpha_i-x)}{(\alpha_i-x)^2}=0.
	\end{split}
	\end{equation}
	Clearly, if $\gamma_i<N\alpha_i$, the function $\frac{\gamma_i-Nx}{\alpha_i-x}$ is a strictly decreasing function of $x$. When $\frac{\gamma_i-Nx}{\alpha_i-x}$ is an integer, say $\frac{\gamma_i-Nx}{\alpha_i-x}=n\le N-1$, choose $h>0$ satisfies $n-1<\frac{\gamma_i-N(x+h)}{\alpha_i-(x+h)}<n$ and $n<\frac{\gamma_i-N(x-h)}{\alpha_i-(x-h)}< n+1$. Then we have $\chi(\frac{\gamma_i-Nx}{\alpha_i-x})=n-1$, $\chi(\frac{\gamma_i-N(x+h)}{\alpha_i-(x+h)})=n-1$ and $\chi(\frac{\gamma_i-N(x-h)}{\alpha_i-(x-h)})=n$. Then one can verify $f_i(x+h)-f_i(x)=\left(\frac{1}{n-1}-\frac{1}{N}\right)\frac{N-n}{n}h>0$ and
	$f_i(x)-f_i(x-h)=\left(\frac{1}{n+1}-\frac{1}{N}\right)\frac{N-n}{n}h\ge 0$. Both $f_i(x+h)-f_i(x)$ and $f_i(x)-f_i(x-h)$ tend to zero when $h\rightarrow 0$. Thus, the function $f_i(x)$ is continuous when $\frac{\gamma_i-Nx}{\alpha_i-x}$ is an integer.
	
	When $\frac{\gamma_i-Nx}{\alpha_i-x}$ is not an integer, $\chi(\frac{\gamma_i-Nx}{\alpha_i-x})\le N-1$ and the function $f_i(x)$ is continuous and differentiable. The derivative of $f_i(x)$ satisfies
	\begin{equation}\label{differential_f_of_x}
	\begin{split}
	f_i'(x)&=\frac{1}{N}-\frac{1}{\chi(\frac{\gamma_i-Nx}{\alpha_i-x})}-\frac{1}{\chi(\frac{\gamma_i-Nx}{\alpha_i-x})+1}+\frac{N}{\chi(\frac{\gamma_i-Nx}{\alpha_i-x})\left[\chi(\frac{\gamma_i-Nx}{\alpha_i-x})+1\right]}\\
	&=\frac{\left[N-\chi(\frac{\gamma_i-Nx}{\alpha_i-x})\right]\left[N-\chi(\frac{\gamma_i-Nx}{\alpha_i-x})-1\right]}{N\chi(\frac{\gamma_i-Nx}{\alpha_i-x})\left[\chi(\frac{\gamma_i-Nx}{\alpha_i-x})+1\right]}\ge 0,
	\end{split}
	\end{equation}
	which means that $f_i(x)$ is a non-decreasing function of $x$. Then we can finally solve the problem (\ref{LP_analytical}) to get the new lower bound
	\begin{equation}\label{tempLabel}
	\lYAT=\min_{x}\left[\sum_{i=1}^N f_i(x)\right]\quad\st\quad \left\{\max_{i}[\gamma_i-(N-1)\alpha_i]\right\}^+\le x\le \min_{i}\frac{\beta_i}{N-1},
	\end{equation}
	where $\lYAT$ denotes the new analytical bound. Since $\sum_{i=1}^N f_i(x)$ is non-decreasing in $x$, defining $\delta=\left\{\max_{i}[\gamma_i-(N-1)\alpha_i]\right\}^+$, the objective value is thus obtained at $x=\delta$ so that $\lYAT=\sum_{i=1}^N f_i(\delta)$.
\end{proof}

\medskip

\section{Comparison of the new analytical bound with the KAT Bound}
We first note by comparing the LP problems of (\ref{LP_KAT}) and (\ref{LP_analytical}) that the new analytical bound is at least as good as the KAT bound. This is because the feasible set of (\ref{LP_KAT}) contains the feasible set of (\ref{LP_analytical}), and both problems (\ref{LP_KAT}) and (\ref{LP_analytical}) share the same objective function. Furthermore, setting $\delta=0$ directly yields  $\lYAT=\lKAT$.

We next quantify the smallest possible improvement of $\lYAT$ over $\lKAT$  via a lower bound on $\lYAT-\lKAT$. We also provide upper and lower bounds on $\lYAT$ in terms of quantities related to de Caen's bound \cite{DeCaen1997}. 
\begin{lemma}\label{lemma_gap}
	A lower bound on $\lYAT-\lKAT$ is given as follows:
	\begin{equation}\label{YA_KAT_gap}
	\lYAT-\lKAT\ge \left\{\sum_{i=1}^{N}\frac{\left[N-\chi(\frac{\gamma_i}{\alpha_i})\right]\left[N-\chi(\frac{\gamma_i}{\alpha_i})-1\right]}{\chi(\frac{\gamma_i}{\alpha_i})\left[\chi(\frac{\gamma_i}{\alpha_i})+1\right]}\right\}\frac{\delta}{N},
	\end{equation}
	where strict inequality for the lower bound (\ref{YA_KAT_gap}) holds if and only if there exists $0<\delta'<\delta$ such that $\frac{\gamma_i-N\delta'}{\alpha_i-\delta'}$ is an integer for some $i\in\{1,\ldots,N\}$.
		
	Furthermore, $\lYAT$ can be bounded as follows:
	\begin{equation}\label{YAT_new_gap}
	\delta+\sum_{i=1}^N\frac{(\alpha_i-\delta)^2}{\gamma_i-N\delta}\le \lYAT\le  \delta+\frac{9}{8}\sum_{i=1}^N\frac{(\alpha_i-\delta)^2}{\gamma_i-N\delta}
	\end{equation}
	where strict inequality for the upper bound holds if and only if $ \frac{\gamma_i-N\delta}{\alpha_i-\delta}\neq\frac{3}{2}$ for some $i$, and where strict inequality for the lower bound holds if and only if $ \frac{\gamma_i-N\delta}{\alpha_i-\delta}$ is not an integer for some $i$.
\end{lemma}

\smallskip

\begin{proof}
	\emph{Lower bound in (\ref{YA_KAT_gap}):} We first prove that $f_i(x)$ is convex in $x$. Note that $f_i(x)$ is a continuous and piecewise differentiable function. However, it is not differentiable when $\frac{\gamma_i-Nx}{\alpha_i-x}$ is an integer. In each interval of $x$ where $\frac{\gamma_i-Nx}{\alpha_i-x}$ is between two successive integers, the derivative of $f_i(x)$ is given by (\ref{differential_f_of_x}) which is positive and only a function of $\chi(\frac{\gamma_i-Nx}{\alpha_i-x})$. Since $\chi(\frac{\gamma_i-Nx}{\alpha_i-x})$ is an integer that does not change in each interval where $\frac{\gamma_i-Nx}{\alpha_i-x}$ is between two successive integers, we only need to show that the derivative of $f_i(x)$ given by (\ref{differential_f_of_x}) is a non-decreasing function of $x$. By denoting $n(x):=\chi(\frac{\gamma_i-Nx}{\alpha_i-x})$, we can write $f_i'(x)=g_i(n)$ where
	\begin{equation}\label{}
	g_i(n):=\frac{(N-n)(N-n-1)}{N(n+1)n}.
	\end{equation}
	Noting that $\gamma_i\le N\alpha_i$, one can verify that $\frac{\gamma_i-Nx}{\alpha_i-x}$ decreases with $x$ and by the definition of $\chi(\cdot)$, $n\le N-1$ and $n=\chi(\frac{\gamma_i-Nx}{\alpha_i-x})$ is a non-increasing function of $x$. Thus, we have $g_i(n)> 0$ and $g_i(n)$ is a decreasing function of $n$ for $1<n\le N-1$, since
	\begin{equation}\label{temp_used_for_equality}
	g_i(n)-g_i(n-1)=\frac{(N-n)(N-n-1)}{N(n+1)n}-\frac{(N-n+1)(N-n)}{Nn(n-1)}< 0,
	\end{equation}
	which implies $f_i'(x)$ is a non-decreasing function of $x$.  Therefore $f_i(x)$ is a convex function of $x$. Finally, by the property of a convex function, we have
	\begin{equation}\label{f_of_i_gap}
	f_i(x)-f_i(0)\ge f_i'(0)(x-0).
	\end{equation}
	Since $\lYAT-\lKAT=\sum_{i} \left[f_i(\delta)-f_i(0)\right]$, by substituting $x=\delta$ into (\ref{f_of_i_gap}) and summing over $i$, the first inequality of (\ref{YA_KAT_gap}) is obtained.
	
	Note that if for all $i=1,\ldots,N$ there does not exist $0<\delta'<\delta$ such that $\frac{\gamma_i-N\delta'}{\alpha_i-\delta'}$ is an integer, the derivative $f_i'(x)=f_i'(0)$ for all $0< x<\delta$ and $i=1,\ldots,N$. Then equality in (\ref{f_of_i_gap}) holds for all $i$, and the first equality holds in (\ref{YA_KAT_gap}). If there exists $0<\delta'<\delta$ such that $\frac{\gamma_i-N\delta'}{\alpha_i-\delta'}$ is an integer for some $i$, then according to  (\ref{temp_used_for_equality}) and the definition of $\chi(\cdot)$, we have $f_i'(\delta')>f_i'(0)$.
	Then, it can be shown that for those $i$ the strict inequality in (\ref{f_of_i_gap}) holds when $x=\delta$. This is because
	\begin{equation}\label{}
	\begin{split}
	f_i(\delta)-f_i(0)&=[f_i(\delta)-f_i(\delta')]+[f_i(\delta')-f_i(0)]\\
	&\ge f_i'(\delta')(\delta-\delta')+f_i'(0)(\delta'-0)\\
	&>f_i'(0)(\delta-\delta')+f_i'(0)(\delta'-0)\\
	&=f_i'(0)(\delta-0).
	\end{split}
	\end{equation}
	Therefore, the first strict inequality in (\ref{YA_KAT_gap}) holds.

\smallskip
	
	\emph{Bounds in (\ref{YAT_new_gap}):}
	 It suffices to show that for any given $i\in\{1,2,\dots,N\}$, and integer $k\in\{1,2,\dots,N-1\}$, we always have
	 \begin{equation}
	 \frac{(\alpha_i-\delta)^2}{\gamma_i-N\delta}\le \frac{a_i(k)}{k}+	\frac{a_i(k+1)}{k+1}\le \frac{9}{8}\frac{(\alpha_i-\delta)^2}{\gamma_i-N\delta},
	 \end{equation}
	 where $a_i(k)+a_i(k+1)=\alpha_i-\delta$ and $k a_i(k)+(k+1)a_i(k+1)=\gamma_i-N\delta$. The lower bound can be obtained directly by the Cauchy Schwarz inequality
	\[
	\frac{a_i(k)}{k}+	\frac{a_i(k+1)}{k+1}\ge\frac{ \left(a_i(k)+a_i(k+1)\right)^2}{k a_i(k)+(k+1)a_i(k+1)}=\frac{(\alpha_i-\delta)^2}{\gamma_i-N\delta},
	\]
	where the inequality is tight if and only if either $a_i(k)$ or $a_i(k+1)$ is zero;  i.e., if and only if $ \frac{\gamma_i-N\delta}{\alpha_i-\delta}$ is an integer for all $i$.
	
	The upper bound can be shown as follows:
	\begin{equation}
	\begin{split}
	\frac{\left[\frac{a_i(k)}{k}+	\frac{a_i(k+1)}{k+1}\right](\gamma_i-N\delta)}{(\alpha_i-\delta)^2}&=\frac{\left[\frac{a_i(k)}{k}+	\frac{a_i(k+1)}{k+1}\right]\left[k a_i(k)+(k+1) a_i(k+1)\right]}{(\alpha_i-\delta)^2}\\
	&=\frac{a_i(k)^2+a_i(k+1)^2+(\frac{k}{k+1}+\frac{k+1}{k})a_i(k)a_i(k+1)}{[a_i(k)+a_i(k+1)]^2}\\
	&=1+\frac{1}{k(k+1)}\frac{a_i(k)a_i(k+1)}{[a_i(k)+a_i(k+1)]^2}\\
	&\le 1+\frac{1}{4}\frac{1}{k(k+1)}\le 1+\frac{1}{8}=\frac{9}{8},
	\end{split}
	\end{equation}
	where the first inequality is tight if and only if $a_i(k)=a_i(k+1)=\frac{\alpha_i-\delta}{2}$ and the second inequality is tight if and only if $k=1$; these are equivalent to $\frac{\gamma_i-N\delta}{\alpha_i-\delta}=\frac{3}{2}$ for all $i$.
\end{proof}

\medskip

\begin{remark}
	Note that when $\delta=0$, $\lYAT=\lKAT$ and (\ref{YAT_new_gap}) in Lemma~\ref{lemma_gap} reduces~to 
$$\sum_{i=1}^N\frac{\alpha_i^2}{\gamma_i}\le \lKAT\le \frac{9}{8}\sum_{i=1}^N\frac{\alpha_i^2}{\gamma_i},$$ 
where $\sum_{i=1}^N\frac{\alpha_i^2}{\gamma_i}$ is just de Caen's bound \cite{DeCaen1997}. In other words, the previously known results that the KAT bound is sharper than de Caen's bound \cite{Kuai2000} and that the KAT bound improves de Caen's bound by a factor of at most $\frac{9}{8}$ \cite{Dembo2000} are recovered.
\end{remark}

\medskip

\section{Numerical Examples}
\begin{table}\caption{System VI.}\label{table_sysVI}
	\centering
	\begin{tabular}{|c|c||c|c|c|c|}
		\hline
		Outcomes $\omega_i$ & $p(\omega_i)$ & $A_1$ & $A_2$ & $A_3$ & $A_4$\\ \hline\hline
		$\omega_0$ & 0.0962  &  $\times$   &      &    &  $\times$ \\ \hline
		$\omega_1$ & 0.0446  &    &     &     &  $\times$\\ \hline
		$\omega_2$ & 0.0581    &    &  $\times$  &     &  $\times$\\ \hline
		$\omega_3$ & 0.0225    & $\times$  &  $\times$   &  $\times$  &   $\times$\\ \hline
		$\omega_4$ & 0.0385    & $\times$ &     &     &   \\ \hline
		$\omega_5$ & 0.0071    & $\times$ &      &  $\times$  &   $\times$\\ \hline
		$\omega_6$ & 0.0582     &    &  $\times$   &      &\\ \hline
	\end{tabular}
\end{table}
\begin{table}\caption{System VII.}\label{table_sysVII}
	\centering
	\begin{tabular}{|c|c||c|c|c|c|}
		\hline
		Outcomes $\omega_i$ & $p(\omega_i)$ & $A_1$ & $A_2$ & $A_3$ & $A_4$\\ \hline\hline
		$\omega_0$ & 0.1832   &  $\times$   &      &    &   \\ \hline
		$\omega_1$ & 0.1219   &    &     &   $\times$  &  \\ \hline
		$\omega_2$ & 0.0337   & $\times$  &  $\times$   &  $\times$   &  $\times$\\ \hline
		$\omega_3$ & 0.0256    &  &  $\times$   &  $\times$  &   \\ \hline
		$\omega_4$ & 0.0682      &    &     &    &   $\times$\\ \hline
		$\omega_5$ &  0.0389    &     &   $\times$ &   $\times$   &  $\times$\\ \hline
		$\omega_6$ & 0.0631     &    &     &   $\times$   & $\times$\\ \hline
	\end{tabular}
\end{table}
\begin{table}\caption{System VIII.}\label{table_sysVIII}
	\centering
	\begin{tabular}{|c|c||c|c|c|c|}
		\hline
		Outcomes $\omega_i$ & $p(\omega_i)$ & $A_1$ & $A_2$ & $A_3$ & $A_4$\\ \hline\hline
		$\omega_0$ & 0.0330       &     &      &    &  $\times$ \\ \hline
		$\omega_1$ & 0.0705   &  $\times$  & $\times$    &   $\times$  &  \\ \hline
		$\omega_2$ & 0.0876     & $\times$  &    &  $\times$   &  $\times$\\ \hline
		$\omega_3$ & 0.0608        &  $\times$  & $\times$  &     & $\times$ \\ \hline
		$\omega_4$ & 0.0865        & $\times$  &  $\times$   &  $\times$  & $\times$   \\ \hline
		$\omega_5$ & 0.0621      &  &   $\times$   &  $\times$   & $\times$   \\ \hline
		$\omega_6$ & 0.0181       &    &  $\times$    &    &   \\ \hline
		$\omega_7$ &  0.0898       &   $\times$   &    &   $\times$   & \\ \hline
		$\omega_8$ & 0.0770       &  &  $\times$     &  $\times$ & \\ \hline
	\end{tabular}
\end{table}
In this section, we evaluate the new lower bounds using eight numerical examples. The first four examples are the same as in \cite{Kuai2000}. The last four examples, Systems V to VIII, are new and are shown in Table \ref{table_sysV}-\ref{table_sysVIII}, respectively. As a reference, the existing DS bound \cite{Dawson1967}, de Caen's bound \cite{DeCaen1997}, the KAT bound (\ref{KAT_bound}) are included for comparison. Furthermore, the GK bound \cite{Gallot1966,Kounias1968}, which exploits full information of all $P(A_i\cap A_j)$'s and $P(A_i)$'s, and the PG bound \cite{Prekopa2005} which exploits $\{P(A_i)\}$, $\{\sum_j P(A_i\cap A_j)\}$ and $\{\sum_{j,l}P(A_i\cap A_j\cap A_l)\}$, are also compared with the new bounds.
The results are shown in Table \ref{table_examples}. The gap of $\lYAT-\lKAT$ and the derived lower bound (\ref{YA_KAT_gap}) are shown in Table \ref{table_gap}.
\begin{table}\caption{Comparison of Lower Bounds (* indicates the bound uses less information, and ** indicates the bound uses more information). }\label{table_examples}
	\small\small
	\centering
	\begin{tabular}{|c||c|c|c|c|c|c|c|c|}
		\hline
		System & $P\left(\bigcup_{i=1}^N A_i\right)$ & DS* & de Caen & KAT & GK** & PG** & Bound (\ref{YA_bound}) & Bound (\ref{LP_optimal})\\ \hline\hline
		I &  0.7890  &  0.7007  &  0.7087  &  0.7247  &  0.7601  &  0.7443 &0.7247  &  0.7487\\ \hline
		II &  0.6740  &  0.6150  &  0.6154  &  0.6227  &  0.6510  &  0.6434 &0.6227  &  0.6398\\ \hline
		III &  0.7890  &  0.6933  &  0.7048  &  0.7222  &  0.7508  & 0.7556 &0.7222  &  0.7427\\ \hline
		IV &  0.9687  &  0.8879  &  0.8757  &  0.8909  &  0.9231  & 0.9148 &0.8909  &  0.9044\\  \hline
		V &  0.3900  &  0.3800  &  0.3495 &   0.3833  &  0.3813 & 0.3900  &0.3900  &  0.3900\\ \hline
		VI & 0.3252   & 0.2706  &  0.2720   & 0.2769   & 0.2972  & 0.3240 &0.3205   & 0.3252\\ \hline
		VII & 0.5346   & 0.3989  &  0.4186  &  0.4434  &  0.4750  & 0.5281 &0.4562   & 0.5090\\ \hline
		VIII & 0.5854  &  0.5395 &   0.5352   & 0.5412  &  0.5390  & 0.5726 &0.5464 &   0.5513\\ \hline
	\end{tabular}
\end{table}
\begin{table}[h!]\caption{Comparison of New Bound (\ref{YA_bound}) with KAT Bound.}\label{table_gap}
	\centering
	\begin{tabular}{|c||c|c|c|c|}
		\hline
		System & KAT & Bound (\ref{YA_bound}) & $\lYAT-\lKAT$ & Lower bound on the gap (\ref{YA_KAT_gap})\\ \hline\hline
		V &  0.3833  &  0.3900  &  0.0067 & 0.0067\\ \hline
		VI & 0.2769   & 0.3205  &   0.0436   &  0.0206\\ \hline
		VII & 0.4434   &  0.4562 &   0.0128  &   0.0128\\ \hline
		VIII &  0.5412  &  0.5464 &     0.0051    &  0.0051\\ \hline
	\end{tabular}
\end{table}

One can see that the KAT bound is at least as good as the DS and de Caen's bounds as already shown in \cite{Kuai2000}. The new bounds are at least as good as the KAT bound in all the examples, as expected. More specifically, the new numerical bound (\ref{LP_optimal}) is sharper than the KAT bound in all examples, and the new analytical bound (\ref{YA_bound}) is sharper than the KAT bound for Systems V to VIII and identical to the KAT bound for Systems I to VI. Concerning the gap of new analytical bound and the KAT bound, the equality of (\ref{YA_KAT_gap}) holds for Systems V, VII and VIII.

Moreover, from the numerical examples, we note that the GK bound \cite{Gallot1966,Kounias1968}, which requires more information (all the information of individual $P(A_i\cap A_j)$ as well as $P(A_i)$'s), is not guaranteed to be sharper than the KAT bound\footnote{Another example in which the GK bound is looser than the KAT bound is given in \cite{Feng2013}.} and the new bounds. For example, the GK bound is worse than the KAT bound as well as the new bounds in Systems V and VIII. It is better than the KAT bound but worse than the new bounds in System VI, better than the KAT bound and the new analytical bound but worse than the new numerical bound in System VII. Furthermore, we note that the PG bound \cite{Prekopa2005}, which also requires more information ($\{P(A_i)\}$, $\{\sum_j P(A_i\cap A_j)\}$ and $\{\sum_{j,l}P(A_i\cap A_j\cap A_l)\}$), is also not necessarily sharper than the new bounds. For example, the PG bound is worse than the new numerical bound in System I and VI.

Finally, we note that all lower bounds considered in this paper can be sharpened algorithmically by optimizing over subsets (e.g., see \cite{Hoppe2006,Behnamfar2007,Hoppe2009}).

\section{Concluding Remarks}
We considered lower bounds on the probability of a finite union of events in terms of the individual event probabilities $\{P(A_i), i=1,\ldots,N\}$ and the sums of the pairwise event probabilities, \emph{i.e.}, $\{\sum_{j:j\neq i} P(A_i\cap A_j), i=1,\ldots,N\}$. An optimal numerical lower bound is obtained by solving an LP problem with $N^2-N+1$ variables, and a new analytical lower bound is established based on solving a relaxed LP problem. It is shown that the new analytical bound is at least as good as the KAT bound.
We conclude with the following remarks: 
\begin{itemize} 
	\item An \emph{optimal numerical upper bound} can be obtained by maximizing the objective function in~(\ref{LP_optimal}), instead of minimizing it, under the same constraints of (\ref{LP_optimal}) and the additional constraint $\sum_{i=1}^N\sum_{k=1}^N\frac{a_i(k)}{k}\le 1$.
	\item 	In the proof of achievability of Theorem \ref{theorem1}, only the last two constraints of (\ref{LP_optimal}) with the additional constraint $\sum_{i=1}^N\sum_{k=1}^N\frac{a_i(k)}{k}\le 1$ are required. Therefore, in other cases where different information is available, optimal lower/upper bounds can be obtained using a similar technique as in  Theorem~\ref{theorem1}.
	\item Finally, we can show that the LP problem (\ref{LP_analytical}) has a unique optimal feasible point.  Therefore, the new analytical bound is achievable if and only if the optimal feasible point of the LP problem (\ref{LP_analytical}) has a corresponding family of events $\{A_i, i=1,\cdots,N\}$ that satisfies the information represented by $\theta=(\alpha_1,\cdots,\alpha_N,\gamma_1,\cdots,\gamma_N)$.
\end{itemize}

\bibliographystyle{siam}
\bibliography{Union_Bounds}



\end{document}